\documentclass[10pt,letterpaper]{article}
\usepackage[T1]{fontenc}
\usepackage[utf8]{inputenc}
\usepackage{lmodern}
\usepackage[margin=1in]{geometry}
\usepackage{amsmath,amssymb,amsthm,mathtools}
\usepackage{booktabs,tabularx,array}
\usepackage{graphicx}
\usepackage{enumitem}
\usepackage{algorithm,algpseudocode}
\usepackage[numbers,sort&compress]{natbib}
\usepackage{microtype}
\usepackage{xcolor}
\usepackage[colorlinks=true,linkcolor=blue!50!black,citecolor=blue!50!black,urlcolor=blue!50!black]{hyperref}
\allowdisplaybreaks[1]
\setlist{itemsep=2pt,topsep=4pt,leftmargin=*}
\newtheorem{theorem}{Theorem}
\newtheorem{lemma}[theorem]{Lemma}

\newtheorem{corollary}[theorem]{Corollary}
\theoremstyle{definition}
\newtheorem{assumption}{Assumption}

\theoremstyle{remark}

\newcommand{\R}{\mathbb R}
\newcommand{\E}{\mathbb E}
\newcommand{\Prob}{\mathbb P}
\newcommand{\Id}{I}
\newcommand{\norm}[1]{\left\lVert#1\right\rVert}
\newcommand{\ip}[2]{\left\langle#1,#2\right\rangle}
\newcommand{\rms}[1]{\norm{#1}_{ L^2}}
\newcommand{\cO}{\mathcal O}
\newcommand{\tO}{\widetilde{\mathcal O}}
\newcommand{\res}{\operatorname{res}}
\newcommand{\zer}{\operatorname{zer}}

\newcommand{\cF}{\mathcal F}
\newcommand{\Alg}{\textsc{Switching Regularization}}
\newcolumntype{Y}{>{\raggedright\arraybackslash}X}

\title{\vspace{-0.5in}\bfseries Optimal Oracle Complexity for\protect\\ Finite-Sum Monotone Inclusions}
\author{Qihao Zhou\thanks{Independent Researcher. Email: \texttt{qihao\_zhou@berkeley.edu}.}}
\date{}

\begin{document}
\maketitle
\thispagestyle{plain}
\begin{abstract}
We present an oracle-optimal method for finite-sum monotone inclusions under mean-square Lipschitz continuity. Our switching regularization method finds a point $y$ and a certificate $g\in G(y)$ with $(\E\norm{F(y)+g}^2)^{1/2}\le\varepsilon$ using $\cO(n+\sqrt n LR/\varepsilon)$ expected component evaluations and resolvent evaluations. It removes the additive $n\log n$ cost of restarting a variance-reduced solver at every regularization stage by switching to a centered stochastic proximal iteration at regularization strength $L/\sqrt n$. Carrying an operator estimate between the remaining stages limits their total cost to $\cO(n)$. A matching $\Omega(n+\sqrt n LR/\varepsilon)$ lower bound holds for randomized linear-span component-oracle algorithms with adaptive stopping and expected query budgets. Thus, for $0<\varepsilon\le LR/2$, our method attains the optimal worst-case expected component complexity in this oracle model, up to universal constants.
\end{abstract}

\section{Introduction}
We consider the finite-sum monotone inclusion
\begin{equation}\label{eq:problem}
0\in T(x):=F(x)+G(x),\qquad F(x)=\frac1n\sum_{i=1}^nF_i(x),
\end{equation}
where $F$ is monotone, its components are mean-square $L$-Lipschitz, and $G$ is maximally monotone. This formulation includes composite convex optimization, convex--concave minimax optimization, and variational inequalities. Given component evaluations and the resolvent of $G$, we seek $y$ and $g\in G(y)$ satisfying $\E\norm{F(y)+g}\le\varepsilon$. Let $R$ bound the distance from the initial point to a solution.

For this problem, \citet{cad24} established a variance-reduced Halpern method with $\tO(n+\sqrt n\,LR/\varepsilon)$ complexity. Recursive regularization provides a useful intermediate bound: applying their strongly monotone solver at every regularization stage gives $\cO(n\log(2+n)+\sqrt n\,LR/\varepsilon)$ complexity; see Appendix~\ref{app:baseline}. The remaining additive $n\log n$ term comes from repeated full-sum initializations after the regularization becomes large. It can dominate the accuracy-dependent term when $LR/\varepsilon\lesssim\sqrt n\log n$.

We remove this overhead by changing solvers as the regularization grows. The method doubles the regularization strength $\lambda$ and uses variance-reduced reflected splitting while $\lambda<L/\sqrt n$. Beyond this threshold, a centered stochastic proximal iteration needs only $\cO(1+L^2/\lambda^2)$ calls per stage. Instead of recomputing a full average at each new center, we update a stored operator estimate with sampled differences. Its error enters the stage operator as a constant shift, which lets us bound both the movement of the roots and the final residual. The later stages together cost $\cO(n)$.

This gives an $\cO(n+\sqrt n\,LR/\varepsilon)$ upper bound with an RMS residual guarantee and a computable certificate in $G(y)$. We prove a matching $\Omega(n+\sqrt n\,LR/\varepsilon)$ lower bound on affine unconstrained instances for randomized linear-span component-oracle algorithms. The construction bounds the residual directly and allows adaptive stopping under expected query budgets. Hence the method is optimal up to universal constants in this model. Table~\ref{tab:main} compares the bounds. The algorithm uses supplied bounds $L,R$ and a multistage schedule; obtaining the same rate with a single-loop method remains open.

\begin{table}[tbp]
\centering\small
\caption{Component complexity for Assumption~\ref{ass:main}. A full evaluation costs $n$ component calls. Upper bounds control the expected inclusion residual. The lower bound assumes $0<\varepsilon\le LR/2$ and the linear-span oracle model of Section~\ref{sec:lower}.}\label{tab:main}
\begin{tabularx}{\textwidth}{@{}Y l Y@{}}
\toprule
Method & Component complexity & Reference\\
\midrule
Full-operator acceleration & $\cO(n+nLR/\varepsilon)$ & \citet{coz24,caizheng23}\\
Variance-reduced Halpern & $\tO(n+\sqrt n\,LR/\varepsilon)$ & \citet{cad24}\\
Recursive regularization, restart-only & $\cO(n\log(2+n)+\sqrt n\,LR/\varepsilon)$ & Appendix~\ref{app:baseline}\\
\textbf{Switching regularization} & $\cO(n+\sqrt n\,LR/\varepsilon)$ & \textbf{Theorem~\ref{thm:upper}}\\
\midrule
\textbf{Lower bound} & $\Omega(n+\sqrt n\,LR/\varepsilon)$ & \textbf{Theorem~\ref{thm:lower}}\\
\bottomrule
\end{tabularx}
\end{table}

\section{Preliminaries}
\subsection{Problem setup}\label{sec:setting}

All norms are Euclidean. For a random vector $X$, write $\rms{X}=(\E\norm{X}^2)^{1/2}$. For a maximally monotone operator $A$, $J_A=(\Id+A)^{-1}$ denotes its resolvent.

\begin{assumption}\label{ass:main}
The operator $F:\R^d\to\R^d$ in \eqref{eq:problem} is monotone, $G:\R^d\rightrightarrows\R^d$ is maximally monotone, and
\begin{equation}\label{eq:msl}
\frac1n\sum_{i=1}^n\norm{F_i(x)-F_i(y)}^2\le L^2\norm{x-y}^2
\qquad\text{for all }x,y\in\R^d.
\end{equation}
There is a point $u^*\in\zer(F+G)$ satisfying $\norm{u_0-u^*}\le R$, where positive bounds $L,R$ are available to the algorithm.
\end{assumption}
Jensen's inequality implies that $F$ is $L$-Lipschitz. Individual components may be nonmonotone. The bound $R$ is supplied to the algorithm, and the complexity scales with this bound.

\paragraph{Accuracy and cost.}
We count each evaluation of a chosen $F_i$ and each evaluation of $J_{\eta G}$ as one call to its respective oracle. The two points in a sampled difference cost two component calls. A full average costs $n$ calls. The output is a pair $(y,g)$ with $g\in G(y)$ and
\begin{equation}\label{eq:target}
\rms{F(y)+g}\le\varepsilon.
\end{equation}
This implies the expected residual guarantee for $\res_T(y)=\inf_{h\in G(y)}\norm{F(y)+h}$. Random full-refresh decisions make the component cost random; our complexity is its expectation. Resolvent arithmetic and vector operations are not charged as component evaluations.

\paragraph{Available regularized resolvents.}
All auxiliary operators have the form $B_{\lambda,b}=G+\lambda\Id-b$. Their resolvents satisfy
\begin{equation}\label{eq:resolvent-identity}
J_{\tau B_{\lambda,b}}(v)
=J_{\frac{\tau}{1+\tau\lambda}G}
\left(\frac{v+\tau b}{1+\tau\lambda}\right).
\end{equation}

\subsection{Examples and observable residuals}
For a composite finite-sum objective, take $F=n^{-1}\sum_i\nabla f_i$ and $G=\partial r$. Only the mean smooth objective needs to be convex; \eqref{eq:msl} controls the component gradients even when individual $f_i$ are nonconvex. The $G$-resolvent is a proximal step for $r$. For a convex--concave saddle function $f(x,y)$ with separable convex penalties $r(x),h(y)$, take
\[
 F(x,y)=\bigl(\nabla_x f(x,y),-\nabla_y f(x,y)\bigr),\qquad
 G(x,y)=\partial r(x)\times\partial h(y).
\]
The skew part of a bilinear interaction is monotone but need not be cocoercive. This is why the algorithm uses mean-square Lipschitz continuity rather than a componentwise cocoercivity condition.

The certificate criterion also controls the usual forward--backward mapping
\[
 P_\eta(u)=\eta^{-1}\bigl(u-J_{\eta G}(u-\eta F(u))\bigr).
\]
For $g\in G(u)$, resolvent nonexpansiveness gives $\|P_\eta(u)\|\le\|F(u)+g\|$. Thus the output certificate bounds an observable mapping norm without a bounded-domain assumption. When $G$ is a normal cone, the mapping uses a projection; when $G=0$, it is exactly $F(u)$.

\subsection{Related work}\label{sec:related}
\paragraph{Monotone residuals and variance reduction.}
Anchoring originates in Halpern iteration \citep{halpern67}, while resolvents and proximal regularization are standard in monotone operator theory \citep{rockafellar76,bauschke17}. Deterministic residual acceleration was developed in \citet{diakonikolas20,yoonryu21,coz24,caizheng23}; \citet{parkryu22} established exact complexity bounds for nonexpansive fixed-point iterations. Reflected splitting handles Lipschitz monotone operators without cocoercivity \citep{malitskytam20}. For finite sums, variance reduction has been studied for saddle-point problems, matrix games, reflected splitting, and variational inequalities \citep{palaniappanbach16,carmon19,alacaoglu21,alacaoglu22,zhanghaskellye22}. \citet{huang26} developed accelerated extra-point methods for finite-sum hemivariational inequalities. The direct residual comparison here is \citet{cad24}; its strongly monotone solver is also an ingredient of our method.

\paragraph{Regularization and finite-sum overhead.}
Recursive regularization for small gradients was developed in \citet{allenzhu18,foster19}, and recursive anchoring for stochastic minimax optimization in \citet{chenluo24}. Our estimator uses sampled differences, as in SARAH, SPIDER, and PAGE \citep{nguyen17,fang18,li21}. Additive finite-sum overhead also appears in convex objective minimization. Varag has an $n\log n$ term in its general convex bound \citep{lanli19}; VRADA reduces this to $n\log\log n$ and attains $\cO(1/n)$ objective accuracy in $\cO(n\log\log n)$ evaluations \citep{song20}. These objective-error guarantees motivate a related complexity question, but use a different accuracy criterion from \eqref{eq:target}.

\paragraph{Other operator and noise assumptions.}
Accelerated root-finding and splitting methods under cocoercivity include \citet{trandinh24,vfosa25}. The control-variate theorem of \citet{vapeg26} requires a positive component-control parameter, and \citet{vrfrbs26} studies reflected estimators for a possibly nonmonotone class. In expectation-oracle settings, stochastic Halpern iteration \citep{csdg22}, recursive anchoring \citep{chenluo24}, variance-growth analyses \citep{alacaoglukim26,alacaoglu26}, and direct stochastic acceleration \citep{yoonloizou26} give guarantees under their respective assumptions. Concurrently, \citet{ji26} apply accumulative regularization to stochastic VIs, obtaining near-optimal residual complexity under bounded variance and additional results under state-dependent noise. Their bounds concern a stochastic expectation oracle, whereas our result exploits exact finite-sum access.

\paragraph{Lower bounds.}
Finite-sum hard instances and information arguments have a substantial history \citep{woodworth16,lanzhou18,han24}. \citet{cad24} already establish near-optimality up to logarithmic factors through existing gap lower bounds. Our construction specializes the chain argument to the residual itself: it uses $G=0$, proves an explicit constant, and handles expected query budgets without logarithmic loss. The oracle class is linear-span component-value access, rather than the stronger proximal incremental oracle considered in \citet{han24}.

\section{Algorithm and Complexity}\label{sec:algorithm}
\subsection{Recursive regularization and the switching threshold}
Set
\begin{equation}\label{eq:parameters}
\begin{gathered}
\mu=\frac{\varepsilon}{4R},\quad \lambda_s=\mu2^s,\quad
S=\max\{0,\lceil\log_2(L/\mu)\rceil\},\\
K=\max\{0,\lceil\log_2(L/(\mu\sqrt n))\rceil\},\quad
x_0=u_0,\quad b_0=\mu u_0.
\end{gathered}
\end{equation}
Then $K\le S$, $\lambda_K\ge L/\sqrt n$, and $\lambda_S\ge L$. After computing $x_{s+1}$, update $b_{s+1}=b_s+\lambda_sx_{s+1}$. Thus the next subproblem adds a regularizer centered at the current approximate solution.

For $s<K$, use the strongly monotone solver of \citet[Theorem 4.6]{cad24} for $F+B_{\lambda_s,b_s}$. A constant-factor RMS distance reduction costs $\cO(n+\sqrt n L/\lambda_s)$, including a full initialization. We call this routine \textsc{VRContract}; its precise implementation appears in Appendix~\ref{app:vr}.

At $s=K$, compute $v_K=F(x_K)$ once. All later stages carry the estimate forward. The switch balances two ways to exploit strong monotonicity. The reflected solver has a finite-sum initialization cost even when the regularized problem is easy. A centered stochastic solve instead costs $\cO(1+L^2/\lambda_s^2)$, which becomes order $n$ at $\lambda_s=L/\sqrt n$ and decreases thereafter. The threshold depends on the component smoothness scale, not just the Lipschitz constant of the mean.

The recursive centers are also essential. The accumulated affine term is
\[
 b_s=\mu u_0+\sum_{j<s}\lambda_jx_{j+1},\qquad
 B_{\lambda_s,b_s}
 =G+\mu(\Id-u_0)+\sum_{j<s}\lambda_j(\Id-x_{j+1}).
\]
Consequently, doubling $\lambda_s$ adds a penalty around the latest approximate solution rather than increasing a penalty around the original point. The centers move less as the subproblems become more strongly monotone. This lets the algorithm reduce local solve accuracy and operator-tracking accuracy together. 

\subsection{The centered stochastic proximal iteration}
Given a center $a$, a stored estimate $v$ of $F(a)$, and $\lambda,b$, start $w_0=a$ and iterate
\begin{equation}\label{eq:centered}
Q_t=v+F_{i_t}(w_t)-F_{i_t}(a),\qquad
w_{t+1}=J_{\tau B_{\lambda,b}}(w_t-\tau Q_t),
\end{equation}
where $i_t$ is uniform and independent of the past, and
\begin{equation}\label{eq:inner-parameters}
\tau=\frac{\lambda}{256(L^2+\lambda^2)},\qquad
N=\left\lceil\frac{\log128}{\tau\lambda}\right\rceil.
\end{equation}
Denote the output $w_N$ by \textsc{CenteredContract}$(a,v,\lambda,b)$. These parameters are used throughout the analysis.

The estimator has conditional mean $F(w_t)+e$, where $e=v-F(a)$. Although $e$ is unknown, it is fixed during this call. The relevant stage root is therefore the zero of $F+B_{\lambda,b}+e$.

This interpretation separates two errors. Sampling within a stage introduces zero-mean fluctuations whose second moment is bounded by $L^2\|w_t-a\|^2$. The stored estimate introduces a persistent shift $e$, rather than a new independent error at every iteration. Strong monotonicity makes the shifted root stable: changing this shift by $h$ moves the root by at most $\|h\|/\lambda$. The analysis therefore tracks the shifted roots and charges estimator drift when the center changes. It does not require a uniform bound on component values or component noise at a solution.

\subsection{Updating the stored estimate and returning a certificate}
After a later stage, draw $m_s=\lceil81L^2/\lambda_s^2\rceil$ fresh independent indices and set
\begin{equation}\label{eq:tracking}
v_{s+1}=v_s+\frac1{m_s}\sum_{j=1}^{m_s}
\big(F_{i_{s,j}}(x_{s+1})-F_{i_{s,j}}(x_s)\big).
\end{equation}
The batches may be sampled with replacement; a batch larger than $n$ still costs only a constant multiple of $n$ at the first later stage.

For $x=x_{S+1}$, perform the exact final correction
\begin{equation}\label{eq:final}
\begin{split}
y&=J_{L^{-1}B_{\lambda_S,b_S}}(x-L^{-1}F(x)),\\
g&=L(x-y)-F(x)-\lambda_Sy+b_S\in G(y).
\end{split}
\end{equation}
The full evaluations of $F(x)$ and $F(y)$ cost $2n$. The second is optional for the guarantee but makes the final residual numerically evaluable.

\begin{algorithm}[tbp]
\caption{\Alg}\label{alg:main}
\begin{algorithmic}[1]
\Require $u_0,n,L,R,\varepsilon$; component and $G$-resolvent oracles
\State Set $\mu,S,K,\lambda_s,x_0,b_0$ according to \eqref{eq:parameters}.
\For{$s=0,\ldots,S$}
  \If{$s<K$}
    \State $x_{s+1}\gets\textsc{VRContract}(x_s,\lambda_s,b_s)$ with RMS factor $1/4$.
  \Else
    \If{$s=K$}
      \State $v_s\gets F(x_s)$ \Comment{One full initialization of the later phase}
    \EndIf
    \State $x_{s+1}\gets\textsc{CenteredContract}(x_s,v_s,\lambda_s,b_s)$ via \eqref{eq:centered}--\eqref{eq:inner-parameters}.
  \EndIf
  \If{$s<S$}
    \State $b_{s+1}\gets b_s+\lambda_sx_{s+1}$.
    \If{$s\ge K$}
      \State Draw a fresh batch of size $m_s=\lceil81L^2/\lambda_s^2\rceil$.
      \State Update $v_{s+1}$ using \eqref{eq:tracking}.
    \EndIf
  \EndIf
\EndFor
\State Compute $(y,g)$ from \eqref{eq:final}; optionally evaluate $F(y)$.
\State \Return $(y,g)$.
\end{algorithmic}
\end{algorithm}

\subsection{How the stages use the oracle}\label{sec:oracle-schedule}
\paragraph{A representative accuracy regime.}
Consider $LR/\varepsilon=\sqrt n$, where the target complexity is linear in the number of components. Then $\mu=L/(4\sqrt n)$ and $K=2$. The first two stages use reflected splitting, and the subsequent strengths are
\[
 \lambda_{K+r}=\frac{2^rL}{\sqrt n},\qquad
 0\le r\le\lceil\log_2\sqrt n\rceil.
\]
The centered iteration and its transition batch consequently have costs proportional to $1+n4^{-r}$. Although there are order $\log n$ later stages, their accuracy-dependent costs decrease geometrically. Their constant per-stage costs also sum to order $n$. For example, $n=10^4$ gives ten stages in total: two reflected stages followed by eight centered stages. The latter share a single full initialization. This schedule illustrates the source of the linear finite-sum overhead; it does not assert that the displayed theoretical constants improve the practical running time.

\paragraph{Why differences, rather than noisy component values, matter.}
A useful invariance clarifies the oracle assumption. Replace each component by $\widehat F_i=F_i+c_i$, where the vectors $c_i$ satisfy $\sum_i c_i=0$. The mean operator and the mean-square Lipschitz constant are unchanged, and
\[
 \widehat F_i(u)-\widehat F_i(v)=F_i(u)-F_i(v).
\]
Every full evaluation and every sampled difference in Algorithm~\ref{alg:main} is therefore unchanged. With the same component indices and refresh decisions, the iterates are identical. Individual component values at a solution may nevertheless grow arbitrarily with the offsets. This explains why the complexity depends on component differences and does not include a bound on the component-value variance at a solution. Using the same index at both points is essential to this cancellation.

\paragraph{State and certification.}
The implementation stores a constant number of vectors: the current iterates, the affine regularizer, and the operator reference used by the active solver. The transition batches can be accumulated as running sums, so no table of $n$ component values is required. Each regularized resolvent is evaluated through \eqref{eq:resolvent-identity}; increasing $\lambda_s$ does not introduce a new oracle. The full evaluations in \eqref{eq:final} are charged separately because they serve a different purpose from initialization: they remove the need to identify the unknown stored-estimate error when constructing the output certificate. The resulting pair $(y,g)$ certifies the original inclusion, while the intermediate iterates solve regularized and, in the later phase, shifted inclusions.

\subsection{Complexity guarantee}\label{sec:results}
\begin{theorem}[Upper bound]\label{thm:upper}
Under Assumption~\ref{ass:main}, Algorithm~\ref{alg:main} returns a computable pair $(y,g)$ satisfying $g\in G(y)$ and
\begin{equation}\label{eq:upper-residual}
\rms{F(y)+g}\le\frac{29}{32}\varepsilon<\varepsilon.
\end{equation}
Its expected numbers of component evaluations and resolvent evaluations are each at most
\begin{equation}\label{eq:upper-cost}
C\left(n+\frac{\sqrt n LR}{\varepsilon}\right)
\end{equation}
for a universal constant $C$. No logarithm of a problem parameter is hidden in $C$.
\end{theorem}

The proof is given in Appendix~\ref{app:upper}. We summarize the estimates that remove the additive overhead. In a later stage, let $e_s=v_s-F(x_s)$ and let $z_s$ be the zero of $F+B_{\lambda_s,b_s}+e_s$. The centered iteration contracts the conditional squared distance to $z_s$ by $1/64$ in $\cO(1+L^2/\lambda_s^2)$ calls. The sampled update between stages and strong monotonicity then give
\begin{equation}\label{eq:overview-invariants}
\rms{x_s-z_s}\le R4^{-s},\qquad
\rms{e_S}\le\frac{\mu R}{4},\qquad
\rms{\lambda_Sx_{S+1}-b_S}\le\frac52\mu R.
\end{equation}
The final correction \eqref{eq:final} converts these estimates to the inclusion residual in Theorem~\ref{thm:upper}.

Since $\lambda_K\ge L/\sqrt n$ and the regularization doubles, the initialization, inner iterations, and batches in the later phase cost
\begin{equation}\label{eq:overview-tail-cost}
n+\sum_{s=K}^{S}\cO\left(1+\frac{L^2}{\lambda_s^2}\right)=\cO(n).
\end{equation}
The early phase costs $\cO(nK+\sqrt nL/\mu)$. If $t=L/(\mu\sqrt n)>1$, then $K\le1+\log_2t$ and $nK=\cO(n+\sqrt nL/\mu)$; otherwise $K=0$. Combining the phases gives the stated complexity bound.

\paragraph{Accuracy regimes.}
Table~\ref{tab:regimes} makes the improvement over restart-only regularization explicit. At $LR/\varepsilon$ of order $\sqrt n$, the target complexity is order $n$: repeated full initializations are then the dominant avoidable cost. At much higher accuracy, the $\sqrt n\,LR/\varepsilon$ term already dominates both upper bounds. The result identifies the finite-sum overhead across these regimes; the asymptotic comparison alone does not imply a smaller practical budget with the displayed constants. Section~\ref{sec:experiments} evaluates tuned implementations separately.

\begin{table}[tbp]
\centering\small
\caption{Dependence of the upper bounds on the accuracy regime, with $n\ge4$ and $0<\varepsilon\le LR/2$. Universal constants are suppressed.}\label{tab:regimes}
\begin{tabular}{@{}lll@{}}
\toprule
Regime & Switching & Restart-only upper bound\\
\midrule
$2\le LR/\varepsilon\le\sqrt n$ & $\cO(n)$ & $\cO(n\log(2+n))$\\
$\sqrt n<LR/\varepsilon<\sqrt n\log(2+n)$ & $\cO(\sqrt n\,LR/\varepsilon)$ & $\cO(n\log(2+n))$\\
$LR/\varepsilon\ge\sqrt n\log(2+n)$ & $\cO(\sqrt n\,LR/\varepsilon)$ & $\cO(\sqrt n\,LR/\varepsilon)$\\
\bottomrule
\end{tabular}
\end{table}

\section{Lower Complexity Bound}\label{sec:lower}
We use component-value access and specialize to $G=0$. The algorithm starts at zero. Every query point and the output must lie in the linear span of previously returned component vectors. Component choices, scalar coefficients, batching, and stopping may depend arbitrarily on the transcript and internal randomness. The proof also applies to the larger, coordinate-based class of zero-respecting algorithms. Dimension may grow with the budget.

\begin{theorem}[Expected-budget lower bound]\label{thm:lower}
Let $n\ge1$, $L,R>0$, and $0<\varepsilon\le LR/2$. Any randomized linear-span component-oracle algorithm guaranteeing expected residual at most $\varepsilon$ on every instance satisfying Assumption~\ref{ass:main} has worst-case expected component cost $M$ satisfying
\begin{equation}\label{eq:lower}
M\ge\frac1{64}\left(n+\frac{\sqrt n LR}{\varepsilon}\right).
\end{equation}
The lower bound holds for $G=0$ and affine monotone components.
\end{theorem}

The proof, given in Appendix~\ref{app:lower}, has two parts. For the accuracy-dependent term, distribute $n$ independent linear chains across the components. Let $D=I-S_m$, where $S_m$ shifts a length-$m$ vector one coordinate to the right, and let $E_i$ embed the $i$th block. The hard family is
\begin{equation}\label{eq:hard-main}
F_i(u)=\frac{\sqrt n L}{2}\,E_i\left(Du^{(i)}-\frac{R}{\sqrt{nm}}e_1\right).
\end{equation}
A component query reveals at most one further coordinate of its chain. A telescoping residual calculation then bounds the error for every adaptive allocation of queries. A constant-probability truncation handles expected budgets.

The additive $n$ term has a separate information argument: a single hidden component contains an unknown signed offset. Until that component is queried, the algorithm cannot identify the sign of the root. Combining the two families gives \eqref{eq:lower}.

\paragraph{How the chain obstructs faster convergence.}
The blocks make information local to the queried component. A response from component $i$ can extend the known prefix of block $i$ by one coordinate, but it does not reveal the remaining blocks. Approximating one chain very accurately therefore leaves an error in the others. For a fixed total budget $T$, choosing the chain length $m=\lfloor2T/n\rfloor+1$ ensures that at least half the blocks remain incomplete. Appendix~\ref{app:lower} obtains the pathwise residual bound
\[
 \|F(u_{\rm out})\|\ge
 \frac{\sqrt n LR}{2\sqrt2(2T+n)}.
\]
The bound holds for every adaptive allocation of the $T$ queries, including allocations based on the observed values. Its dependence on $\sqrt n$ comes from normalizing the component differences by the mean-square constant $L$ while keeping the total solution distance equal to $R$. It is not obtained by replacing a gap tolerance with a residual tolerance.

Expected query budgets require an additional step because the algorithm may stop at a random time. If its expected cost is at most $M$, the probability of using at most $2M$ queries is at least one half. Applying the pathwise construction at that budget and retaining this constant-probability event yields an expected-residual obstruction. The additive information bound supplies $M\ge n/4$, allowing the two terms to be combined with a universal constant. The proof therefore accommodates adaptive stopping without a confidence-dependent logarithm. The construction may increase the dimension with the budget; the conclusion concerns worst-case component-query complexity in the stated oracle class, not arithmetic complexity in a fixed dimension.

\begin{corollary}[Optimality in the stated oracle model]\label{cor:optimal}
For $0<\varepsilon\le LR/2$, the worst-case expected component complexity is
\[
\Theta\left(n+\frac{\sqrt n LR}{\varepsilon}\right)
\]
within the randomized linear-span component-oracle model that also permits the specified $G$-resolvent outputs. In particular, the upper-bound algorithm is linear-span on the lower-bound subclass $G=0$.
\end{corollary}
The parameter $L$ is the mean-square Lipschitz constant. Appendix~\ref{app:lower} records the construction's different scaling under average cocoercivity.

\section{Numerical Experiments}\label{sec:experiments}
\begin{figure}[!htbp]
\centering
\includegraphics[width=0.94\linewidth]{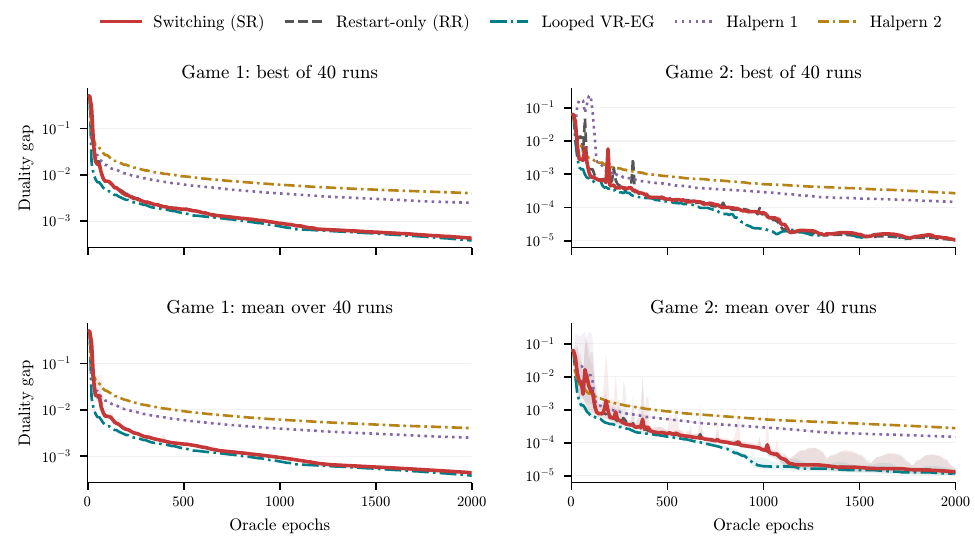}
\caption{Top: for each method and game, the complete trajectory with the smallest final duality gap among 40 runs. Bottom: means over all 40 runs, with the full observed range shaded. Best-run curves are selected outcomes, not estimates of expected performance. Lower is better; one oracle epoch costs $n$ component calls.}\label{fig:experiments}
\end{figure}

We compare tuned switching regularization (SR), its restart-only ablation (RR), looped variance-reduced extragradient (VR-EG), and Algorithms 1 and 2 of \citet{cad24} on two $500\times500$ simplex games from \citet{alacaoglu22}. Each method uses the same fixed budget and 40 independent sampling runs. Parameters are frozen after separate validation. Figure~\ref{fig:experiments} presents both the best observed run of each method and the complete-run average. One oracle epoch is $n$ component evaluations, including full initializations and refreshes. The tuned SR/RR schedules differ from Theorem~\ref{thm:upper}; Appendix~\ref{app:checks} gives the setup.

On Game 2, SR's best observed duality gap is $1.06\times10^{-5}$, compared with $1.05\times10^{-5}$ for RR and $1.02\times10^{-5}$ for looped VR-EG. The best SR gap is $13.9$ and $25.3$ times smaller than the best gaps of the two Halpern baselines. Averaging all 40 runs gives $1.33\times10^{-5}$ for SR and $1.18\times10^{-5}$ for looped VR-EG. Looped VR-EG also leads on Game 1. Thus the best observed SR runs are competitive, while the aggregate comparison does not establish a consistent practical advantage from switching.

\section{Conclusion}
We establish the optimal worst-case expected component complexity $\Theta(n+\sqrt n LR/\varepsilon)$ for finite-sum monotone inclusions in the randomized linear-span oracle model, for $0<\varepsilon\le LR/2$. Carrying an operator estimate through the strongly regularized stages removes repeated full-sum initializations. The matching residual lower bound allows adaptive stopping and expected query budgets.

Further work could seek a single-loop method that retains this rate without requiring supplied bounds on $L$ and $R$. Sharper control of the inner solves may also reduce the conservative constants and support practical parameter choices with provable guarantees. On the lower-bound side, extending the construction beyond linear-span algorithms would clarify whether the same complexity is unavoidable under unrestricted component-value queries.\label{page:main-end}

\typeout{MAIN-TEXT-PAGES=\getpagerefnumber{page:main-end}}
\clearpage

\section*{Acknowledgements}

Codex with GPT-6 Astra and Claude Science Opus 5.5 were used to assist with the development of the construction and proofs, numerical checks of algebraic identities, the review of related literature, and the preparation and revision of this manuscript. The author is responsible for the mathematical claims, references, and final presentation.

\typeout{REFERENCES-START-PAGE=\number\value{page}}
\begingroup
\raggedright
\urlstyle{same}
\bibliographystyle{plainnat}
\bibliography{references}
\endgroup

\clearpage
\typeout{APPENDIX-START-PAGE=\number\value{page}}
\appendix
\setlength{\parskip}{2pt}
\setlength{\abovedisplayskip}{8pt plus 2pt minus 3pt}
\setlength{\belowdisplayskip}{8pt plus 2pt minus 3pt}
\setlength{\belowdisplayshortskip}{4pt plus 2pt minus 2pt}
\section{Operator Facts and the Early-Stage Solver}\label{app:vr}
\subsection{Resolvents and perturbations of roots}
Under Assumption~\ref{ass:main}, $F$ is continuous, monotone, and defined on all of $\R^d$. Consequently $T=F+G$ is maximally monotone. For every $\lambda>0$ and every vector $c$, $T+\lambda\Id-c$ has a unique zero, namely $J_{T/\lambda}(c/\lambda)$. Both $J_A$ and $\Id-J_A$ are nonexpansive when $A$ is maximally monotone. If $A$ is $\lambda$-strongly monotone, $J_{\tau A}$ is $(1+\tau\lambda)^{-1}$-Lipschitz.

We also use the following elementary perturbation estimate. Suppose $A$ is $\lambda$-strongly monotone, $0\in A(z)$, and $0\in A(z')+h$. Strong monotonicity gives
\[
\lambda\norm{z-z'}^2\le\ip{h}{z-z'}\le\norm{h}\norm{z-z'},
\]
and hence $\norm{z-z'}\le\norm{h}/\lambda$. These facts apply pathwise to operators whose affine terms are measurable functions of previous samples.

To verify \eqref{eq:resolvent-identity}, the relation $y=J_{\tau(G+\lambda\Id-b)}(v)$ is equivalent to
\[
v+\tau b\in(1+\tau\lambda)y+\tau G(y).
\]
Dividing by $1+\tau\lambda$ yields the claimed identity. This also verifies the oracle implementation of every subsolver below.

\subsection{A fixed-factor specialization of variance-reduced reflected splitting}
\begin{lemma}[Published strong-monotonicity subsolver]\label{lem:early}
Under Assumption~\ref{ass:main}, let $z$ be the zero of $F+B_{\lambda,b}$. There is a routine \textnormal{\textsc{VRContract}} starting at $a$ such that
\[
\E[\norm{a^+-z}^2\mid a,b]\le\frac1{16}\norm{a-z}^2
\]
at expected cost $\cO(n+\sqrt n L/\lambda)$. The resolvent count satisfies the same order bound.
\end{lemma}
\begin{proof}
Apply \citet[Theorem 4.6]{cad24} with $A=F$, $B=B_{\lambda,b}$, and strong monotonicity parameter $\lambda$. For $n\ge2$, initialize $u^0=w^0=w^{-1}=a$, set
\[
p=\frac1n,\qquad \eta=\frac{\sqrt{p(1-p)}}{2L},\qquad
M_\lambda=\left\lceil14\left(n+\frac{\sqrt n L}{\lambda}\right)\log(4\sqrt6)\right\rceil,
\]
and for $t=0,\ldots,M_\lambda-1$ draw a uniform component $i_t$ and compute
\begin{align*}
u^{t+1}&=J_{\eta B_{\lambda,b}}\left((1-p)u^t+pw^t
-\eta\big[F(w^t)-F_{i_t}(w^{t-1})+F_{i_t}(u^t)\big]\right),\\
w^{t+1}&=\begin{cases}u^{t+1},&\text{with probability }p,\\w^t,&\text{with probability }1-p.\end{cases}
\end{align*}
The refresh decision is independent of the component draw and past history. Cache $F(w^t)$, computing a full average only when the reference point changes. Return $a^+=u^{M_\lambda}$. The cited theorem's linear estimate gives the stated distance contraction. One full initialization and at most $2+pn=3$ expected component calls per iteration give the cost claim. The same estimate applies after conditioning on the data of an outer stage. If $n=1$, use two identical copies of $F_1$; this only changes universal constants. If $a=z$, the initialized method remains at the root, so no division by a zero initial distance is needed.
\end{proof}

\section{Proof of the Upper Bound}\label{app:upper}
\subsection{The centered contraction lemma}
\begin{lemma}\label{lem:centered}
Fix $a,v,b\in\R^d$ and $\lambda>0$. Let $e=v-F(a)$ and let $z$ be the zero of
\[
\widetilde H=F+B_{\lambda,b}+e.
\]
The iteration \eqref{eq:centered} with parameters \eqref{eq:inner-parameters} satisfies
\[
\E[\norm{w_N-z}^2\mid a,v,b]\le\frac1{64}\norm{a-z}^2.
\]
It uses at most $2N$ component evaluations and $N$ resolvent evaluations, where $N=\cO(1+(L/\lambda)^2)$.
\end{lemma}
\begin{proof}
Condition on the stage data $a,v,b$ throughout. Write $D=\norm{a-z}$ and $r=\tau\lambda$. At iteration $t$, write $Q_t=F(w_t)+e+\xi_t$. Conditional on the inner history $\cF_t$,
\begin{equation}\label{eq:noise-bound}
\E[\xi_t\mid\cF_t]=0,\qquad
\E[\norm{\xi_t}^2\mid\cF_t]\le L^2\norm{w_t-a}^2.
\end{equation}
Indeed, $\xi_t$ is the centered version of $F_{i_t}(w_t)-F_{i_t}(a)$, and removing its mean cannot increase its second moment.

Since $z=J_{\tau B_{\lambda,b}}(z-\tau(F(z)+e))$, the resolvent contraction gives
\begin{align}
\E[\norm{w_{t+1}-z}^2\mid\cF_t]
&\le\frac{1}{(1+r)^2}\E\left[\norm{w_t-z-\tau(F(w_t)-F(z))-\tau\xi_t}^2\mid\cF_t\right]\notag\\
&\le\frac{(1+\tau^2L^2)\norm{w_t-z}^2+\tau^2L^2\norm{w_t-a}^2}{(1+r)^2}\notag\\
&\le\frac{(1+3\tau^2L^2)\norm{w_t-z}^2+2\tau^2L^2D^2}{(1+r)^2}.
\label{eq:centered-recursion}
\end{align}
The second inequality uses monotonicity and Lipschitz continuity of the mean $F$, as well as \eqref{eq:noise-bound}; it does not use component monotonicity.

The choice of $\tau$ yields
\[
0<r\le\frac1{256},\qquad \tau^2L^2\le\frac r{256}.
\]
Thus $3\tau^2L^2\le3r/256\le r-r^2-r^3$, implying
\[
\frac{1+3\tau^2L^2}{(1+r)^2}\le1-r,
\qquad
\frac{2\tau^2L^2}{(1+r)^2}\le\frac r{128}.
\]
Taking expectations in \eqref{eq:centered-recursion} and iterating gives
\[
\E\norm{w_N-z}^2
\le\left((1-r)^N+\frac1{128}\right)D^2
\le\frac{D^2}{64},
\]
where $(1-r)^N\le e^{-rN}\le1/128$. Finally,
\[
\frac1r=256\left(1+\frac{L^2}{\lambda^2}\right).
\]
Each sampled difference costs at most two component calls. The logarithm in $N$ is the fixed numerical constant $\log128$.
\end{proof}

\subsection{The filtration and root-tracking induction}
At the start of stage $s$, condition on all completed stages and all their samples. In a later stage, $x_s,v_s,b_s$ and
\[
e_s=v_s-F(x_s)
\]
are then fixed. The unique zero $z_s$ of $\widetilde H_s=F+B_{\lambda_s,b_s}+e_s$ is measurable with respect to these stage data. For $s\le K$, define $e_s=0$; this agrees with the stored estimate at $s=K$ because $v_K=F(x_K)$. All inner samples are fresh. The batch in \eqref{eq:tracking} is sampled after $x_{s+1}$ has been computed, independently conditional on the complete history through that point.

\begin{lemma}[Stage invariants]\label{lem:invariants}
Let $D_s=\rms{x_s-z_s}$, $\delta_s=\rms{x_{s+1}-z_s}$, and $\bar D_s=R4^{-s}$. Then
\begin{equation}\label{eq:stage-bounds}
D_s\le\bar D_s,\qquad
\delta_s\le\begin{cases}D_s/4,&s<K,\\D_s/8,&s\ge K,\end{cases}
\end{equation}
and $\rms{e_S}\le\mu R/4$.
\end{lemma}
\begin{proof}
The initial root is $z_0=J_{T/\mu}(u_0)$. Since the displacement map $\Id-J_{T/\mu}$ is nonexpansive and vanishes at $u^*$,
\[
D_0=\norm{u_0-J_{T/\mu}(u_0)}\le\norm{u_0-u^*}\le R.
\]
Lemmas~\ref{lem:early} and \ref{lem:centered}, applied conditional on each stage's data and then averaged, give the displayed bounds on $\delta_s$.

For $s<K$, including the transition $s=K-1$ if it occurs, the update of $b_s$ gives
\[
\widetilde H_{s+1}=\widetilde H_s+\lambda_s(\Id-x_{s+1}),\qquad
z_{s+1}=J_{\widetilde H_s/\lambda_s}(x_{s+1}).
\]
Here and below $\Id-a$ denotes $x\mapsto x-a$. The point $z_s$ is a fixed point of $J_{\widetilde H_s/\lambda_s}$. Nonexpansiveness of the corresponding displacement map yields, pathwise,
\[
\norm{x_{s+1}-z_{s+1}}\le\norm{x_{s+1}-z_s}.
\]
Consequently $D_{s+1}\le\delta_s\le D_s/4$.

For $K\le s<S$, define $\Delta_s=e_{s+1}-e_s$. The estimator update implies
\begin{equation}\label{eq:delta-definition}
\Delta_s=\frac1{m_s}\sum_{j=1}^{m_s}
\big[F_{i_{s,j}}(x_{s+1})-F_{i_{s,j}}(x_s)\big]
-\big[F(x_{s+1})-F(x_s)\big].
\end{equation}
Let $\cF_{s,+}$ be the history through $x_{s+1}$ before this fresh batch. Then
\[
\E[\Delta_s\mid\cF_{s,+}]=0,\qquad
\E[\norm{\Delta_s}^2\mid\cF_{s,+}]
\le\frac{L^2}{m_s}\norm{x_{s+1}-x_s}^2.
\]
Since $\rms{x_{s+1}-x_s}\le D_s+\delta_s\le9D_s/8$, the batch size gives
\begin{equation}\label{eq:bias-increment}
\rms{\Delta_s}\le\frac{9L}{8\sqrt{m_s}}D_s
\le\frac{\lambda_sD_s}{8}\le\frac{\lambda_s\bar D_s}{8}.
\end{equation}
Introduce the intermediate root
\[
\bar z_{s+1}=J_{\widetilde H_s/\lambda_s}(x_{s+1}).
\]
Its distance from $x_{s+1}$ is at most $\norm{x_{s+1}-z_s}$. Moreover,
\[
\widetilde H_{s+1}
=\widetilde H_s+\lambda_s(\Id-x_{s+1})+\Delta_s.
\]
The operator before the constant shift $\Delta_s$ is $2\lambda_s$-strongly monotone. The perturbation estimate in Appendix~\ref{app:vr} therefore gives
\[
\norm{z_{s+1}-\bar z_{s+1}}\le\frac{\norm{\Delta_s}}{2\lambda_s}.
\]
By Minkowski's inequality,
\[
D_{s+1}\le\delta_s+\frac{\rms{\Delta_s}}{2\lambda_s}
\le\frac{D_s}{8}+\frac{D_s}{16}\le\frac{D_s}{4}.
\]
Induction proves \eqref{eq:stage-bounds}. Since $e_K=0$, one more use of Minkowski's inequality gives
\begin{equation}\label{eq:total-bias}
\rms{e_S}\le\sum_{s=K}^{S-1}\rms{\Delta_s}
\le\frac{\mu R}{8}\sum_{s=K}^{S-1}2^{-s}
\le\frac{\mu R}{4}.
\end{equation}
This argument does not assume independence between stage roots or between accumulated stage errors. In particular, conditional centering of $\Delta_s$ is not used to set $e_s$ to zero.
\end{proof}

\subsection{The accumulated regularizer}
\begin{lemma}\label{lem:regularizer}
For $x=x_{S+1}$,
\[
\rms{\lambda_Sx-b_S}\le\frac52\mu R.
\]
\end{lemma}
\begin{proof}
Lemma~\ref{lem:invariants} gives
\[
\rms{x_{s+1}-x_s}\le\frac54R4^{-s}.
\]
Summing these increments from $s$ onward yields
\begin{equation}\label{eq:tail-displacements}
\rms{x-x_s}\le\frac53R4^{-s}.
\end{equation}
The definitions of $\lambda_s,b_s$ imply the exact identity
\[
\lambda_Sx-b_S=\mu(x-u_0)+\sum_{j=0}^{S-1}\lambda_j(x-x_{j+1}).
\]
Using \eqref{eq:tail-displacements},
\[
\rms{\lambda_Sx-b_S}
\le\frac53\mu R+\frac53\mu R\sum_{j=0}^{S-1}2^j4^{-(j+1)}
\le\frac53\mu R+\frac56\mu R=\frac52\mu R.
\]
\end{proof}

\subsection{A genuine inclusion-residual certificate}
Let $x=x_{S+1}$, $\lambda=\lambda_S$, $b=b_S$, $z=z_S$, $e=e_S$, and $\delta=\delta_S$. By the final-stage contraction,
\begin{equation}\label{eq:last-distance}
\lambda\delta\le\frac18\mu R2^{-S}\le\frac18\mu R.
\end{equation}
The point $z$ satisfies
\[
z=J_{L^{-1}B_{\lambda,b}}\big(z-L^{-1}(F(z)+e)\big).
\]
Compare this with the definition of $y$ in \eqref{eq:final}. Since $\lambda\ge L$,
\[
\rms{y-z}\le\frac{2L\delta+\rms{e}}{L+\lambda}
\le\delta+\frac{\rms{e}}{L+\lambda},
\qquad
\rms{x-y}\le2\delta+\frac{\rms{e}}{L+\lambda}.
\]
The resolvent relation gives $g=L(x-y)-F(x)-\lambda y+b\in G(y)$. Hence
\begin{align*}
\rms{F(y)+g}
&\le 2L\rms{x-y}+\rms{\lambda y-b}\\
&\le(2L+\lambda)\rms{x-y}+\frac52\mu R\\
&\le(4L+2\lambda)\delta+\frac{2L+\lambda}{L+\lambda}\rms{e}+\frac52\mu R\\
&\le6\lambda\delta+\frac32\rms{e}+\frac52\mu R\\
&\le\left(\frac34+\frac38+\frac52\right)\mu R
=\frac{29}{32}\varepsilon.
\end{align*}
The penultimate line uses $\lambda\ge L$; the last uses \eqref{eq:total-bias} and \eqref{eq:last-distance}. This proves the accuracy part of Theorem~\ref{thm:upper}.

\subsection{Oracle complexity}
The expected early-phase cost is at most a universal constant times
\[
nK+\sum_{s=0}^{K-1}\frac{\sqrt n L}{\mu2^s}
\le nK+\frac{2\sqrt n L}{\mu}.
\]
If $t=L/(\mu\sqrt n)>1$, then $K\le1+\log_2t\le1+t/\ln2$, so $nK=\cO(n+\sqrt n L/\mu)$. If $t\le1$, there is no early phase.

In the later phase, the number of inner iterations is bounded by
\[
\sum_{s=K}^{S}\left[1+256\log128\left(1+\frac{L^2}{\lambda_s^2}\right)\right],
\]
and the number of component calls for updating the stored estimate is at most
\[
2\sum_{s=K}^{S-1}\left(1+81\frac{L^2}{\lambda_s^2}\right).
\]
Moreover,
\[
\sum_{s=K}^{S}\frac{L^2}{\lambda_s^2}\le\frac43\frac{L^2}{\lambda_K^2}\le\frac43n,
\qquad
S-K+1\le1+\lceil\log_2\sqrt n\rceil=\cO(\log(2+n)).
\]
Since $\log(2+n)=\cO(n)$, the stage count is absorbed into the order-$n$ bound. If $\lambda_K\ge L$, the later phase contains only one stage. Its initial full evaluation costs $n$. The final correction and residual evaluation cost $2n$ additional component calls and one resolvent. Combining all terms gives
\[
\E N_F,\ \E N_G=\cO\left(n+\frac{\sqrt n L}{\mu}\right)
=\cO\left(n+\frac{\sqrt n LR}{\varepsilon}\right).
\]
Here $N_F$ counts component values and $N_G$ counts $G$-resolvents. This completes the proof of Theorem~\ref{thm:upper}.

\section{Proof of the Expected-Budget Lower Bound}\label{app:lower}
\subsection{Oracle restrictions}
For a general $G$, the linear-span class permits each oracle input and output iterate to lie in the span of the initial point, past component values, and past $G$-resolvent outputs. For this appendix, $G=0$ and $u_0=0$, so the resolvent cannot enlarge this span. Each query point and the returned point must belong to the linear span of previously returned component vectors. An algorithm may adaptively choose indices, coefficients, and its stopping time using its entire transcript and internal randomness. The proof only needs the weaker condition that a point never has a nonzero coordinate that has not appeared in a previous oracle response; this is the zero-respecting condition. Every full average is charged as $n$ component queries. The resolvent of $G=0$ is the identity and supplies no new information.

Let $M$ be a uniform upper bound on the algorithm's expected component cost over the admissible instances. If no finite $M$ exists, the theorem is immediate. The following two constructions use the same parameters $n,L,R$ but need not be combined into one instance.

\subsection{An information lower bound for the additive term}
Choose a uniform hidden index $J\in\{1,\ldots,n\}$ and an independent uniform sign $\theta\in\{-1,+1\}$. In one dimension, set
\begin{equation}\label{eq:hidden}
F_i(x)=Lx-nL\theta R\,\mathbf1_{\{i=J\}}.
\end{equation}
Every component is monotone and $L$-Lipschitz, and $F(x)=L(x-\theta R)$ has root $\theta R$. Thus the solution distance is $R$ and \eqref{eq:msl} holds. The large component offset is allowed because the assumption controls differences, not the variance of component values at a point.

Until index $J$ is queried, every response is $Lx$. Conditional on missing $J$, the sign is still uniform and the output cannot depend on it. Therefore, for every such output,
\begin{equation}\label{eq:hidden-error}
\frac L2\big(|x-R|+|x+R|\big)\ge LR.
\end{equation}
To bound discovery probability with adaptive stopping, fix the algorithm's internal randomness and consider its no-discovery transcript. Let $q\le n$ be the number of distinct indices it would query before stopping, or before exhausting all indices. Repeated queries can only increase cost. Its discovery probability is $q/n$, whereas its expected cost over $J$ is at least
\[
\frac1n\sum_{j=1}^q j+\frac{n-q}{n}q
=q-\frac{q(q-1)}{2n}\ge\frac q2.
\]
This also covers an unbounded no-discovery path by truncation; finite expected cost rules out a positive-probability infinite actual run. Averaging over the algorithm's randomness gives discovery probability at most $2M/n$. By \eqref{eq:hidden-error},
\begin{equation}\label{eq:hidden-lower}
\E|F(x_{\rm out})|\ge LR\left(1-\frac{2M}{n}\right).
\end{equation}
The expectation here also averages over the hard distribution. An algorithm guaranteeing expected residual at most $\varepsilon\le LR/2$ on every instance must therefore satisfy
\begin{equation}\label{eq:n-lower}
M\ge n/4>0.
\end{equation}
This part of the proof does not require a span restriction.

\subsection{The block-chain family}
Let $m\ge1$ and partition $u\in\R^{nm}$ into $n$ blocks $u^{(i)}\in\R^m$. Define
\[
S_m(z_1,\ldots,z_m)=(0,z_1,\ldots,z_{m-1}),\qquad D=I-S_m,
\]
and let $E_i$ embed a vector in block $i$. With
\[
\alpha=\frac{\sqrt n L}{2},\qquad a=\frac{R}{\sqrt{nm}},
\]
set
\begin{equation}\label{eq:chain}
F_i(u)=\alpha E_i(Du^{(i)}-ae_1).
\end{equation}
Because $\norm{S_m}\le1$,
\[
\ip{Dz}{z}=\norm{z}^2-\ip{S_mz}{z}\ge0,
\]
so each component is monotone. Also $\norm D\le2$ and
\[
\frac1n\sum_i\norm{F_i(u)-F_i(v)}^2
=\frac{\alpha^2}{n}\sum_i\norm{D(u^{(i)}-v^{(i)})}^2
\le L^2\norm{u-v}^2.
\]
The unique root satisfies $u^{*(i)}=a\mathbf1_m$, since $D\mathbf1_m=e_1$. Its norm is $R$.

A component response has support only in its own block. Applying $D$ can introduce at most the next coordinate of that block, and the offset only introduces its first coordinate. By induction, after $q_i$ queries to component $i$, any zero-respecting output has block-$i$ support contained in the first $q_i$ coordinates, unless the entire block is already revealed. This property holds for every realization of an adaptive randomized algorithm.

\begin{lemma}[Exact restricted-chain minimum]\label{lem:chain}
If $z$ is supported on the first $q<m$ coordinates, then
\[
\norm{Dz-ae_1}^2\ge\frac{a^2}{q+1}.
\]
Equality is attained by $z_j=a(q+1-j)/(q+1)$ for $1\le j\le q$ and zero elsewhere.
\end{lemma}
\begin{proof}
The first $q+1$ coordinates of $Dz-ae_1$ sum to $-a$ because the differences telescope and $z_{q+1}=0$. Cauchy--Schwarz gives the lower bound. For the displayed $z$, these $q+1$ coordinates are each $-a/(q+1)$, and the remaining coordinates vanish.
\end{proof}

\subsection{A pathwise bound for a fixed query budget}
Fix a real budget $T\ge0$ and choose $m=\lfloor2T/n\rfloor+1$. If the total number of component queries is at most $T$, at least $n/2$ blocks have $q_i\le\lfloor2T/n\rfloor=m-1$; otherwise the total number of queries exceeds $T$. On these blocks, Lemma~\ref{lem:chain} gives squared residual at least $a^2/m$ before scaling. The mean operator's block scaling is $\alpha/n=L/(2\sqrt n)$. Hence
\begin{align}
\norm{F(u_{\rm out})}^2
&\ge\frac{L^2}{4n}\frac n2\frac{a^2}{m}
=\frac{L^2R^2}{8nm^2},\notag\\
\norm{F(u_{\rm out})}
&\ge\frac{LR}{2\sqrt{2n}\,m}
\ge\frac{\sqrt n LR}{2\sqrt2(2T+n)}.
\label{eq:fixed-lower}
\end{align}
This is a pathwise statement, with no averaging over an algorithm's randomness.

\subsection{Expected budgets and the explicit constant}
For the uniform expected budget $M>0$, use the chain length $m=\lfloor4M/n\rfloor+1$. Let $Q$ be the algorithm's random component count on this fixed instance. Since $\E Q\le M$, Markov's inequality gives $\Prob(Q\le2M)\ge1/2$. Applying \eqref{eq:fixed-lower} on this event gives
\begin{equation}\label{eq:expected-chain}
\E\norm{F(u_{\rm out})}\ge\frac{\sqrt n LR}{4\sqrt2(4M+n)}.
\end{equation}
Therefore, a uniform expected-residual guarantee at tolerance $\varepsilon$ requires
\[
\frac{\sqrt n LR}{\varepsilon}\le4\sqrt2(4M+n).
\]
From \eqref{eq:n-lower}, $n\le4M$, so
\[
\frac{\sqrt n LR}{\varepsilon}\le32\sqrt2 M,
\qquad
n+\frac{\sqrt n LR}{\varepsilon}\le(4+32\sqrt2)M<64M.
\]
This proves Theorem~\ref{thm:lower}. A worst-case algorithm must handle both hard families; their combination does not require them to share the same instance or dimension.

\subsection{Why the smoothness normalization matters}
For the shift matrix,
\[
2\ip{Dz}{z}-\norm{Dz}^2=\norm z^2-\norm{S_mz}^2\ge0.
\]
Consequently the family \eqref{eq:chain} is average $1/L_c$-cocoercive with $L_c=2\alpha=\sqrt n L$. For $m\ge2$, the ratio is attained by $z=e_1$. Thus this lower bound cannot be reinterpreted as a cocoercive lower bound with the same numerical constant $L$. Similarly, a component resolvent can reveal several chain coordinates at once, so the component-value argument does not prove a lower bound for that stronger oracle.

\section{Restarted Regularization Baseline}\label{app:baseline}
The full-operator bound in Table~\ref{tab:main} follows from Theorem~\ref{thm:upper} with one component, charging $n$ calls per full evaluation.

Using \textsc{VRContract} at every stage gives $e_s=0$, $D_s\le R4^{-s}$, $\delta_s\le D_s/4$, and $\rms{\lambda_Sx-b_S}\le5\mu R/2$. Hence the final correction satisfies
\[
\rms{F(y)+g}\le6\lambda_S\delta_S+\frac52\mu R
\le\left(\frac32+\frac52\right)\mu R=\varepsilon.
\]
Thus this baseline meets tolerance $\varepsilon$. Its expected cost is
\[
\cO\left(n\log(2+LR/\varepsilon)+\sqrt n\,LR/\varepsilon\right)
\subseteq\cO\left(n\log(2+n)+\sqrt n\,LR/\varepsilon\right).
\]
For the second bound, split into $LR/\varepsilon\le n$ and $LR/\varepsilon>n$. In the latter case, $\log(2+LR/\varepsilon)\le\log(2+n)+\log(LR/(n\varepsilon))$, and $n\log(LR/(n\varepsilon))\le LR/\varepsilon\le\sqrt n\,LR/\varepsilon$.

\section{Experimental Setup}\label{app:checks}
\paragraph{Problems and measurements.}
We solve $\min_{x\in\Delta_n}\max_{y\in\Delta_n}y^\top Ax$ with $n=500$, $x_0=y_0=\boldsymbol1/n$, and the matrices released by \citet{alacaoglu22} (one-based indices):
\[
 A^{(1)}_{ij}=\frac{i+j-1}{2n-1},\qquad
 A^{(2)}_{ij}=\left(\frac{|i-j|+1}{2n-1}\right)^2.
\]
We use Euclidean projection $\Pi$ onto $\Delta_n\times\Delta_n$ and the uniform paired oracle $F_i(x,y)=(nA_{i,:}^{\top}y_i,-nA_{:,i}x_i)$, with $L^2=n\max\{\max_i\|A_{i,:}\|^2,\max_j\|A_{:,j}\|^2\}$. We measure the duality gap $\max_i(Ax)_i-\min_j(A^\top y)_j$. The residual theorem does not bound duality gap; the oracle differs from the source paper's importance-sampled and entropic variants.

\paragraph{Selection and accounting.}
We select parameters by mean final gap over two independent 200-epoch validation runs, then evaluate 40 independent runs per method and game at 2,000 epochs on fixed matrices. The top panels show each method's best final-gap run, selected after evaluation; the bottom panels show the all-run mean and range. Selected runs do not estimate typical performance; shading is the observed range, not a confidence interval. All methods search step multipliers $q\in\{1/8,1/4,1/2,1,2,4,8,16,32\}$; SR/RR also search inner multipliers $a\in\{0.2,1,4,16\}$. Full sums cost $n$ calls and sampled differences cost two. All methods exclude validation and diagnostic calls and retain the preceding output when an update exceeds the budget. An earlier five-run comparison favored looped VR-EG over EG, EAG, loopless VR-EG, VR-FoRB, and Carmon's Euclidean method in mean gap.

\paragraph{Implementations.}
Looped VR-EG uses loop length $n$, mixing weight $1-1/n$, and the latest completed loop average. Halpern 1 uses the released PAGE schedule. Halpern 2 reports completed outer iterates, uses $\max\{1,\lfloor0.05n\log(k+2)\rfloor\}$ inner steps and $c=\sqrt n/\|A\|_F$, and refreshes to the new inner iterate. Base steps are $0.99/(\sqrt nL)$ for SR/RR and looped VR-EG, $0.3/L$ for Halpern 1, and $5\sqrt{p(1-p)}/(\sqrt n+1)$ for Halpern 2, where $p=1/n$. Selected $q$ values for (SR, RR, VR-EG, Halpern 1, Halpern 2) are $(32,32,32,32,32)$ on Game 1 and $(16,16,16,32,2)$ on Game 2; SR/RR use $a=0.2$. Some selections lie at the grid boundary.

SR/RR warm-start sweep $j$ with $\mu_j=(L/\sqrt n)10^{-\min\{j,12\}}$, retaining Algorithm~\ref{alg:main}'s stage strengths and switch threshold but using
\[
 M_s=\lceil a\max\{n,(\gamma\lambda_s)^{-1}\}\rceil,\quad
 N_s=\lceil a(1+L^2/\lambda_s^2)\rceil,\quad
 m_s=\lceil4L^2/\lambda_s^2\rceil.
\]
The centered step is $\min\{\gamma,\lambda_s/[2(L^2+\lambda_s^2)]\}$, followed at each sweep's end by an exact forward--backward correction. RR restarts every stage. The theorem does not cover these parameter choices.

Earlier comparisons on \citet{cad24}'s matrix game and quadratic problem favored VR-EG and EAG, respectively, without a consistent SR advantage over RR. On 27 affine configurations, SR improved on both Halpern baselines, but VR-EG and RR had lower pooled median target-crossing budgets.

\end{document}